\documentclass[a4paper,oneside,12pt]{article}

\usepackage{amsmath,amsfonts,amscd,amssymb}
\usepackage{longtable,geometry}
\usepackage[english]{babel}
\usepackage[active]{srcltx}
\usepackage[T1]{fontenc}
\usepackage{graphicx}
\usepackage{pstricks}
\usepackage{bbm}
\usepackage{mathtools}
\usepackage{hyperref}
\usepackage{scalerel}

\usepackage{MnSymbol}
\usepackage{stmaryrd}
\usepackage{nicefrac}
\usepackage{calrsfs}
\usepackage{enumitem}
\usepackage[mathscr]{eucal}

\usepackage{xcolor}
\usepackage{framed}

\colorlet{shadecolor}{blue!15}

\usepackage{amsthm}
\newtheorem{theorem}{Theorem}
\newtheorem{conj}{Conjecture}
\newtheorem{claim}{Claim}
\newtheorem{corollary}{Corollary}
\newtheorem{lemma}{Lemma}
\newtheorem{proposition}{Proposition}

\newtheorem{remark}[theorem]{Remark}

\newtheorem{question}[conj]{Question}
  
\newcommand{\calA}{\mathcal{A}}
\newcommand{\calB}{\mathcal{B}}

\newcommand{\calP}{\mathcal{P}}

\newcommand{\bbP}{\mathbb{P}}

\newcommand{\bbR}{\mathbb{R}}

\newcommand{\bbZ}{\mathbb{Z}}

\newcommand{\sfG}{{\mathsf{G}}}

\newcommand{\sfB}{{\mathsf{B}}}
\newcommand{\bfn}{\mathbf{n}}

\newcounter{cst}

\numberwithin{equation}{section}

\newcommand{\rk}[1]{\bgroup\color{red}%
  \par\medskip\hrule\smallskip%
  \noindent\textbf{#1}%
  \par\smallskip\hrule\medskip\egroup}

\newcommand{\notleftrightnn}  {\mathrel{\ooalign{$\xleftrightarrow{\substack{ \phantom{+}\bfn \phantom{+}}}$\cr\hidewidth$\scaleobj{0.7}{/}$\hidewidth}}}

\newcommand{\notleftrightxy}  {\mathrel{\ooalign{$\xleftrightarrow{\substack{\bfn_{[xy]}}}$\cr\hidewidth$\scaleobj{0.7}{/}$\hidewidth}}}

\allowdisplaybreaks
\title{Translation-Invariant Gibbs States of Ising model:\\ General Setting}
\author{Aran Raoufi}
\date{\today}

\begin{document}
\maketitle
 
 \begin{abstract}
We prove that at any inverse temperature $\beta$ and on any transitive amenable graph, the automorphism-invariant Gibbs states of the ferromagnetic Ising model are convex combinations of the plus and minus states. 
This is obtained for a general class of interactions, that is automorphism-invariant and irreducible coupling constants. 
The proof uses the random current representation of the Ising model. 
The result is novel when the graph is not $\bbZ^d$, or when the graph is $\bbZ^d$ but endowed with infinite-range interactions, or even $\bbZ^2$ with finite-range interactions.

Among the corollaries of this result, we can list continuity of the magnetization at any non-critical temperature, the differentiability of the free energy, and the uniqueness of FK-Ising infinite-volume measures.
 \end{abstract}
 \section{Introduction} 
\paragraph{Motivation and history of the problem.}
Let $\sfG$ be a countable locally finite transitive graph. Let $(J_{xy})_{x,y\in \sfG}$ be a family of nonnegative real numbers. 
For any $\tau \in \{-1,0, +1\}^{\sfG }$, define the Hamiltonian of the ferromagnetic Ising model on a finite subset $\Lambda \subset \sfG$ 
\begin{equation*} \label{hamil}
H_{\Lambda,\tau}(\sigma) := -\sum_{x, y \in \Lambda} J_{xy} \sigma_x \sigma_y  -   \sum_{\substack{x \in \Lambda \\y \in \sfG \setminus \Lambda}} J_{xy}\sigma_x \tau_x
\end{equation*}
for any $\sigma \in \{-1, +1\}^\Lambda$. 
For $\beta \in [0, \infty)$, the Ising measure on $\Lambda$ with boundary condition $\tau$ at inverse temperature $\beta$ is the probability measure $ \langle.\rangle_{\Lambda, \beta}^\tau$ on configurations $\sigma \in  \{-1, +1\}^\Lambda$ such that
$$ \langle \sigma \rangle_{\Lambda, \beta}^\tau = \frac {\exp (-\beta H_{\Lambda,\tau}(\sigma))} {Z(\Lambda, \beta, \tau)},$$
where $Z(\Lambda, \beta, \tau)$ is a normalizing constant. Let $\langle . \rangle_{\Lambda, \beta}^+$ (respectively $\langle . \rangle_{\Lambda, \beta}^-$, $\langle . \rangle_{\Lambda, \beta}^0$) denote the case when $\tau \equiv 1$ (respectively $\tau \equiv -1$, $\tau \equiv 0$).

In their seminal work, Dobrushin and Lanford-Ruelle \cite{DobGibs,LanRuelle} set a general framework to study Gibbs states of statistical physics lattice models on the infinite lattices, which are equilibrium states of the infinite system.
For the Ising model at inverse temperature $\beta$, the set of Gibbs states of a system is equal to the set of probability measures $\mu$ on $\{+1,-1\}^\sfG$ satisfying the DLR equation, which states that for any finite $\Lambda \subset G$ and any $f: \{ -1, + 1\}^{\Lambda} \rightarrow \bbR$,
\begin{equation*} \label{eq:DLR}
\mu[f]= \int_{\tau \in \{+1,-1\}^G}  \langle f \rangle_{\Lambda, \beta}^\tau \, \,d\mu (\tau) .
\end{equation*}
see \cite{georgii2011gibbs, friedli2016statistical}.

The set of Gibbs states is non-empty. The two states $\langle . \rangle_{G,\beta}^+$ and $\langle . \rangle_{G,\beta}^-$ can be obtained by taking weak limits of  respectively the measures  $\langle . \rangle_{\Lambda_n,\beta}^+$ and $\langle . \rangle_{\Lambda_n,\beta}^-$, where $(\Lambda_n)_{n \geq 1}$ is a family of finite sets exhausting $\sfG$. It is easy to show that these two measures are extremal Gibbs states. They are called the \emph{plus} and the \emph{minus} states. Another Gibbs state obtainable as the weak limit of Ising measures on finite boxes is $\langle . \rangle_{G,\beta}^0$ which is the weak limit of   $\langle . \rangle_{\Lambda_n,\beta}^0$. Although $\langle . \rangle_{G,\beta}^0$ is translational invariant it is not necessarily extremal for all values of $\beta$.

Understanding the model is entangled with understanding the Gibbs states of the system. It is known that the set of Gibbs states for the Ising model is a simplex (this in fact can be shown for a very general class of systems)  \cite{friedli2016statistical}, and hence the question of identifying the Gibbs states is equivalent to the question of understanding the extremal Gibbs states. 

The focus of this article is translational invariant Gibbs states. Lebowitz, in his pioneering work \cite{lebowitz1977coexistence}, studied translational 
invariant Gibbs states, and proved that there exist at most countably many $\beta$ for which there exist extremal translation-invariant states other 
than $\langle . \rangle_{G,\beta}^+$ and $\langle . \rangle_{G,\beta}^-$. Lebowitz's method is robust and works for general interactions and any 
amenable graph, but the drawback is that it does not provide any control or information over those countably many values of $\beta$ where uniqueness may fail. 

An important progress was later made in works of Aizenman \cite{aizenman1980translation} and Higuchi \cite{higuchi1981absence} (see \cite{coquille2012finite} for a new proof) who proved that on $\bbZ^2$ with the nearest-neighbor interaction,  the only extremal measures are the pure plus and minus states. The result of Aizenman and Higuchi is not valid for $\bbZ^d$ with $d \geq 3$ and the question of identifying all the extremal Gibbs states remains widely open. 

Dobrushin \cite{dobrushin1973gibbs} found examples of extremal Gibbs states which are not plus or minus states. Nevertheless, Dobrushin's examples were not translation-invariant, and it was indeed still expected that the only translation-invariant Gibbs states were the convex combinations of the plus and minus states. 

Note that when $\beta < \beta_c$ (see below for the definition of $\beta_c$), it is easy to show that $\langle . \rangle_{G,\beta}^+ = \langle . \rangle_{G,\beta}^-$ and the set of Gibbs states is a singleton. Hence, the main problem is understanding the Gibbs states for $\beta \geq \beta_c$.  The advancement in understanding the Gibbs states of Ising model on $\bbZ^d$ for $\beta= \beta_c$ resulted from works that proved that the phase transition is continuous, that is $\langle . \rangle_{G,\beta}^+ = \langle . \rangle_{G,\beta}^-$ for $\beta=\beta_c$. 
This was done for $d \geq 4$ in the eighties by Aizenman and Fernandez \cite{AizFer86}. 
The continuity of the phase transition for the case of $d=3$  is proved by Aizenman, Duminil-Copin, Sidoravicius in \cite{AizDumSid15}.

Below the critical temperature, for any $\beta>\beta_c$ and for finite-range interaction on $\bbZ^d$, Bodineau \cite{bodineau2006translation} proved that the only extremal translation-invariant Gibbs states are the plus and minus states. 

Here, we extend results concerning translation-invariant Gibbs states to a generalized setting. We go further than $\bbZ^d$ and study Ising models on transitive amenable graphs with a very general set of interaction. 
We prove that any automorphism-invariant Gibbs states is the convex combination of the plus and minus states (see Theorem \ref{thm:main}). 
The motivation for studying statistical physics model on general graphs goes beyond generalizing the results on $\bbZ^d$ to other graphs.
Studying statistical physics model on general graphs often leads to proofs that only rely on the coarse geometric structure of the graph and hence sheds more light on the behavior of the models. Also, by proving the result in the general setting, we settle some previously unknown case for $\sfG=\bbZ^d$, namely

\begin{itemize}
\item $d=2$ with interactions which are not nearest-neighbor,
\item any $d$ and the system has infinite-range interaction.
\end{itemize}

Let us finish this subsection by highlighting one special case we find especially worthy of attention: when $G=\bbZ$ and $J_{xy}= |x-y|^{-2}$ the phase transition is not continuous, that is the magnetization is positive at $\beta_c$ \cite{aizenman1988discontinuity}. However, perhaps surprisingly, our result implies that there exists only two extremal translation-invariant states. That is, the discontinuity of magnetization does not necessarily imply that $\langle . \rangle^0_{\sfG, \beta_c} \neq \tfrac 1 2 \langle . \rangle^+_{\sfG, \beta_c} +  \tfrac 1 2 \langle . \rangle^-_{\sfG, \beta_c}$.


\paragraph{Statement of the result.} Let $\sfG$ be a countable locally finite graph with edge set $E(\sfG)$. We slightly abuse the notation and let  $\sfG$ also denote the vertex set of the graph. $\sfG$ is \emph{amenable} if 
$$\inf_{A \subset \sfG, \: \vert A \vert < \infty} \frac{\vert \partial A \vert}{\vert A \vert} = 0,$$
where $\partial A = \{ x \in \sfG \setminus A:  \exists y \in A, \,  \{x,y\} \in E(\sfG) \}.$ 
Let $\mathrm{Aut}(\sfG)$ denote the automorphism group of $\sfG$. From now on, fix a vertex-transitive subgroup $\Gamma$ of $\mathrm{Aut}(\sfG)$.

We consider the Ising model on $\sfG$ with coupling constants $( J_{xy})_{x,y \in \sfG}$ satisfying the following conditions:
\begin{itemize}[noitemsep, nolistsep]
\item[\textbf{(C1)}]
Ferromagnetic: For any $x,y \in \sfG$, $J_{xy} \geq 0$.
\item[\textbf{(C2)}]
Invariance under automorphisms of $\Gamma$: For any automorphism $\varphi \in \Gamma$, $J_{\varphi(x) \varphi(y)} = J_{xy}$.
\item[\textbf{(C3)}]
Irreducibility: For any $x,y \in \sfG$, there exist $v_1,v_2, \dots, v_k \in \sfG$ such that
$$ J_{xv_1} J_{v_1v_2} \dots J_{v_{k-1}v_k} J_{v_k y} > 0 .$$
\item[\textbf{(C4)}]
Local finiteness: For any $x \in \sfG$, $$\sum_{y \in G} J_{xy} < \infty.$$
\end{itemize}
Let $\mu$ be a measure on $\{+1, -1\} ^\sfG$. For $\varphi \in \Gamma$ an automorphism of $\sfG$, and an event $A \subset \{+1, -1\} ^ \sfG $ define 
$$\varphi ^{-1}(A) = \{\sigma \circ \varphi ^{-1} : \sigma \in A \},$$
where $\circ$ is composition of the functions.
The measure $\mu$ is called \emph{$\Gamma$-invariant} if for any event $A$ and any $\varphi \in \Gamma$, $\mu [A] = \mu [\varphi^{-1} (A)]$.
\begin{theorem}  \label{thm:main}
Let $\sfG$ be an amenable transitive graph endowed with coupling constants satisfying conditions \emph{(C1-C4)}. Then, for any $\Gamma$-invariant Gibbs state $\mu$ associated with the Ising model with $\beta \geq 0$, there exists $\lambda  \in [0,1]$ such that
$$ \mu = \lambda \, \langle . \rangle_\beta^+ + (1 -\lambda) \, \langle . \rangle_\beta^-.$$
\end{theorem}

Lebowitz, in his classical work \cite{lebowitz1977coexistence}, reduced the problem to the analysis of the two-point function of the free and plus states. He showed that for a transitive graph endowed with coupling constants satisfying {(C1-C4)}, the followings are equivalent:
\begin{itemize} [nolistsep, noitemsep]
\item
Every automorphism-invariant Gibbs state of Ising model, at inverse temperature $\beta$, is a convex combination of the plus and minus states.
\item 
For any $x,y \in \sfG$ satisfying $J_{xy}>0$, $\langle \sigma_x \sigma_y \rangle ^+_\beta =\langle \sigma_x \sigma_y \rangle ^0_\beta$.
\end{itemize} 
Hence, our goal in this article is to prove the following proposition.
\begin{proposition} \label{prop:semimain}
Let $\sfG$ be an amenable transitive graph endowed with coupling constants satisfying \emph{(C1-C4)}. For any $\beta \in [0, \infty]$ and $x , y \in \sfG$ such that $J_{xy} >0$,
\begin{equation} \label{eq:main}
\langle\sigma_x \sigma_y\rangle  ^+_{\sfG, \beta} = \langle \sigma_x \sigma_y \rangle ^0_{\sfG, \beta}.
\end{equation}
\end{proposition}

Proposition \ref{prop:semimain} has several other consequences.
Fix an origin $0 \in \sfG$, the quantity $\langle \sigma_0 \rangle _{G, \beta}^+ $ is of special interest and is called \emph{spontaneous magnetization}. 
The critical point of the model is defined as
$$ \beta_c := \inf \, \lbrace\beta:  \langle \sigma_0 \rangle_{G, \beta}^+  > 0 \rbrace. $$
Understanding the continuity of the spontaneous magnetization as a function of $\beta$ is an important problem. Theorem \ref{thm:main} has the following corollary.
\begin{corollary} \label{cor:magcon}
For an amenable transitive graph $\sfG$ endowed with coupling constants satisfying conditions \emph{(C1-C4)}, $\beta \rightarrow \langle \sigma_0 \rangle_{\sfG,\beta}^+$ is continuous everywhere except possibly at $\beta_c$. 
\end{corollary} 
The already mentioned discontinuity of phase transition for the case of $\bbZ$ and $J_{xy}= |x-y|^{-2}$ shows that the continuity at $\beta_c$ is not valid for any choice of coupling constants.
It is expected that the magnetization is continuous at $\beta_c$  for amenable graphs when endowed with nearest-neighbor coupling constants and satisfying $\beta_c < \infty$.
This is so far proved for $\bbZ^d$ \cite{Yan52,AizFer86,AizDumSid15}, amenable graphs of exponential growths \cite{raoufi2016note}, and general bi-periodic two-dimensional graphs \cite{duminil-lis2017}.
See Question \ref{q:nonamcr} for a discussion on the non-amenable case. 
 
\medbreak
Another object of interest is the free energy. Fix $\tau \in \{+1, 0, -1 \}^{\sfG}$, and fix a sequence of finite subsets $\Lambda_n$, satisfying $\tfrac{|\partial \Lambda_n|}{|\Lambda_n|} \rightarrow 0$, exhausting $\sfG$. Define the \emph{free energy}
$$\psi(\sfG, \beta) := \lim_{n \rightarrow \infty} \, \frac{1}{| \Lambda_n|} \log Z(\Lambda_n, \beta, \tau) . $$
It is straightforward to show that the limit  $\psi$ exists, and for amenable graphs it does not depend on the choice of the boundary condition $\tau$ (we refer the reader to \cite{friedli2016statistical}).
\begin{corollary} \label{cor:pressure}
For an amenable transitive graph $\sfG$ endowed with coupling constants satisfying conditions \emph{(C1-C4)}, the function $\psi(\sfG, \beta)$ is differentiable for any $\beta \in [0, \infty]$.
\end{corollary}

\paragraph{Consequences for the FK-Ising model.}
Let $\calP_2(\mathsf G)=\big \lbrace \{x,y\} \subset \sfG \big\rbrace$. For $\Lambda \subset \sfG$, define $\calP_2(\mathsf G, \Lambda) = \calP_2(\sfG) \setminus \calP_2(\Lambda)$.
Fix $\xi \in \{0,1\}^{\calP_2(\mathsf G,\Lambda)}$. 
For  $\Lambda \subset \sfG$ a finite subgraph of $\sfG$, the \emph{FK-Ising} model at inverse temperature $\beta$ with boundary condition $\xi$ is a measure on $\{ 0,1\}^{\calP_2(\mathsf \Lambda)}$ defined as
$$
	\phi^\xi_{\Lambda, \beta} [\omega] = 
	\frac{ 2^{k^\xi(\omega)}  \prod_{ x,y \in \sfG} (e^{2\beta J(x,y)}-1) ^{\omega(x,y)}  }{ Z_{\textrm{FK}}(\Lambda, \beta, \xi)},
	$$ 
where $k^\xi(\omega)$ is equal to the number of connected components of the graph with vertex set $\sfG$ and edge set $\omega \cup \xi$ that intersect $\Lambda$. The term $Z_{\textrm{FK}}(\Lambda, \beta, \xi)$ is a normalizing constant. 
Let $\phi_{\Lambda, \beta}^1$ (respectively $\phi_{\Lambda, \beta}^0)$ denote the case when $\xi \equiv 1$ (respectively $\xi \equiv 0$). 
The FK-Ising model is related to the Ising model through a coupling giving
$$\langle \sigma_x \sigma_y \rangle_{\Lambda, \beta} ^0 =  \phi^0_{\Lambda, \beta} [x \leftrightarrow y], \qquad   \langle \sigma_x \sigma_y \rangle_{\Lambda, \beta} ^+ =  \phi^1_{\Lambda, \beta} [x \leftrightarrow y].$$
We refer the reader to \cite{Gri06, duminil2017lectures} for more details on FK-Ising.

Just like for the Ising model, one can define the set of infinite-volume FK-Ising measures using the spatial Markov property. The set of infinite-volume FK-Ising measures is the set of probability measures $\nu$ on $\{0,1\}^{\calP_2(\sfG)}$ such that for any finite $\Lambda \subset \sfG$ and any $f:\{0,1\}^\Lambda \rightarrow \bbR$ they satisfy
\begin{equation*} 
\nu[f]= \int_{\xi \in \{0,1\}^{\calP_2(\sfG,\Lambda)}}  	\phi^\xi_{\Lambda, \beta} [f] \, \,d\nu (\xi) .
\end{equation*}

\begin{corollary} \label{cor:fk}
Let $\sfG$ be an amenable transitive graph. The set of FK-Ising infinite-volume measures with coupling constants satisfying the conditions \emph{(C1-C4)} has only one element. In particular, $\phi^0_{\sfG, \beta} = \phi^1_{\sfG, \beta}$.
\end{corollary}

\paragraph{About the proof.} 
The proof is based on the random current representation of the Ising model. 
More precisely, we will study the geometrical properties of the currents' infinite-volume limit. 
In \cite{AizDumSid15}, the authors utilized the infinite-volume random current measure to prove continuity of magnetization at criticality for $\bbZ^d$.  They proved that if  $\inf_x \, \langle \sigma_0 \sigma_x\rangle^0 = 0$, then there exists a unique Gibbs state. The assumption of $\inf_x \, \langle \sigma_0 \sigma_x\rangle^0 = 0$ was relevant for them because their study focused on $\beta_c$, and there this assumption was guaranteed by the infra-red bound in the case of $\bbZ^d$. They showed that uniqueness of Gibbs states is equivalent to nonexistence of infinite cluster for the double random currents. 

We also work with infinite-volume random currents, but our study relies on more involved properties than the existence or nonexistence of the infinite cluster. 
Our approach controls geometrical properties of the infinite cluster of the double current. We prove that $\langle . \rangle_{\sfG, \beta}^+ \neq \langle . \rangle_{\sfG, \beta}^0$ forces certain geometry on the infinite-volume currents, and then we show that this geometry is paradoxical.

Having Proposition \ref{prop:semimain} at hand, the proof of the corollaries is already in the literature: we refer the reader to \cite{grimmett1995stochastic}.  
Section 2 gathers the required preliminaries related to the random current representation. Proposition \ref{prop:semimain} is proved in the third section.  Finally in the last section we discuss some open questions.

\paragraph{Acknowledgment.}
This research was supported by an IDEX grant from Paris-Saclay. This research was started when the author was a PhD student at the University of Geneva and was supported by the NCCR SwissMAP, the ERC AG COMPASP, and the Swiss NSF.
The author is thankful to H. Duminil-Copin, A. Van Enter, and Y. Velenik for many helpful comments on the manuscript.

\section{Random Current Representation}
A \emph{current} on a graph $\Lambda$ (finite or infinite) is a function from pairs of vertices to non-negative integers. Let $\Omega_\Lambda$ denote the set of all currents on $\Lambda$, that is $\Omega_\Lambda$ $:= \{0,1,2,...\}^{\calP_2(\Lambda)}$. 
For a current $\bfn \in \Omega_\Lambda$, define the sources of $\bfn$ to be
$$\partial \bfn := \big \lbrace x \in \Omega_\Lambda : \sum_{y \in \Lambda} \bfn_{xy} \text{ is odd} \big \rbrace.$$
For finite $\Lambda \subset \sfG$ and $\beta \geq 0$, the weight of the current $\bfn \in \Omega_\Lambda$ is defined by
$$\omega_\beta(\bfn) := \prod_{x,y \in \Lambda} \frac{(\beta J_{xy})^{\bfn_{xy}}}{\bfn_{xy}!}.$$
Many quantities of the Ising model can be expressed in terms of different combinatorial sums on weighted currents. 
This, together with the switching lemma (Lemma \ref{lem:switch}), made the currents a very effective tool in the study of the Ising model \cite{Aiz82,AizFer86,shlosman1986signs, AizBarFer87, sakai2007lace, AizDumSid15, DumTas15}.
Here we only describe the expression of two-point correlations in terms of the current.
For any $x,y \in \Lambda$,
\begin{equation} \label{eq:iscor}
\langle \sigma_x \sigma_y \rangle^0_{\Lambda, \beta} = \frac{\displaystyle \sum_{\bfn \in \Omega_\Lambda, \, \partial \bfn = \{ x,y\}} \omega_\beta(\bfn)}{\displaystyle \sum_{\bfn \in \Omega_\Lambda, \, \partial \bfn = \emptyset} \omega_\beta(\bfn)}.
\end{equation}
Similarly, the correlation for the Ising model with plus boundary condition can be expressed in terms of currents. To achieve this, one modifies the graph $\Lambda$ as follows: define $\Lambda \cup \delta$ to be the graph obtained by adding a vertex $\delta$ to the set of vertices of $\Lambda$, and setting $J_{x \delta } = \sum_{y  \in \sfG \setminus \Lambda} J_{xy}$ for any $x \in \Lambda$. For any $x,y \in \Lambda$,
\begin{equation} \label{eq:ispcor}
\langle \sigma_x \sigma_y \rangle^+_{\Lambda, \beta} = \frac{\displaystyle \sum_{\bfn \in \Omega_{\Lambda\cup \delta}, \, \partial \bfn = \{ x,y\}} \omega_\beta(\bfn)}{\displaystyle \sum_{\bfn \in \Omega_{\Lambda \cup \delta}, \, \partial \bfn = \emptyset} \omega_\beta(\bfn)}.
\end{equation}

For $\bfn \in \Omega_\Lambda$, define the \emph{multigraph associated with $\bfn$} to be the graph with vertex set $\Lambda$ and $\bfn_{xy}$ edges between any two vertices $x$ and $y$. 
The connected components of the multigraph are called \emph{clusters}. A current $\bfn$ belongs to the event $\{x \stackrel{\bfn}{\longleftrightarrow} y\}$ 
if $x$, and $y$ are in the same cluster of the multigraph associated with $\bfn$. For $\Lambda' \subset \Lambda$, the event $\{x \stackrel{\bfn}{\longleftrightarrow} y$ in $\Lambda'\}$ is the event that there exists a path between $x$ and $y$ in the multigraph associated with $\bfn$ using only vertices in $\Lambda'$.


A very useful tool in the study of the currents is the switching lemma \cite{GriHurShe70}.
\begin{lemma}[Switching lemma] \label{lem:switch}
Let $\Lambda_1$ and $\Lambda_2$ be such that $\Lambda_1 \subset \Lambda_2$. Assume that $\Lambda_1$ and 
$\Lambda_2$ are respectively endowed with coupling constants $J_1$ and $J_2$ such that $J_2\big |_{\Lambda_1} = J_1$. For any $F:\Omega_{\Lambda} \rightarrow \bbR$, $\{x,y\} \in \Lambda_1$ and $A\subset \Lambda_2$,\vspace{0.5cm}
\begin{align*}
\displaystyle\sum_{\substack{\bfn_1 \in \Omega_{\Lambda_1}, \, \partial \bfn_1 = \{x,y \} \\  \bfn_2 \in \Omega_{\Lambda_2}, \, \partial \bfn_2 = A}} &\omega_\beta(\bfn_1) \, \omega_\beta(\bfn_2) \, F(\bfn_1+\bfn_2) \\
&= \displaystyle\sum_{\substack{\bfn_1 \in \Omega_{\Lambda_1}, \, \partial \bfn_1 = \emptyset \\  \bfn_2 \in \Omega_{\Lambda_2}, \, \partial \bfn_2 = A \triangle \{ x,y\}}} \omega_\beta(\bfn_1) \, \omega_\beta(\bfn_2) \, F(\bfn_1+\bfn_2) \, {\bf1} [x \xleftrightarrow{\bfn_1+\bfn_2} y \emph{\text{ in }} \Lambda_1 ].
\end{align*}
\end{lemma}
We now define a probability measure on the set of currents. For a finite graph $\Lambda$ and $A \subset \Lambda$, a probability measure on the set of currents with source $A$ can be defined with the law 
$$ \bbP_{\Lambda,\beta}^{A}(\bfn) = \frac{\omega_\beta(\bfn) \, {\bf1}_{\partial \bfn = A}}{\displaystyle\sum _{\bfn \in \Omega_\Lambda, \: \partial \bfn = A}  \omega_\beta (\bfn)}. $$
For ease of notation, $ \bbP_{\Lambda,\beta}$ denotes $\bbP^\emptyset_{\Lambda,\beta}$. Combining \eqref{eq:iscor} with Lemma \ref{lem:switch} gives us 
\begin{equation}\label{eq:iscor2}\langle \sigma_x \sigma_y \rangle_{\Lambda, \beta}^2 =  \bbP_{\Lambda,\beta} \otimes \bbP_{\Lambda,\beta} [x \xleftrightarrow{\bfn_1+\bfn_2} y].
\end{equation}
Equations \eqref{eq:iscor} and \eqref{eq:iscor2} show that the infinite-volume weak limits of the measures $\bbP_{\Lambda, \beta}$ and $\bbP_{\Lambda\cup \delta, \beta}$ are useful for studying $\langle  \sigma_x \sigma_y \rangle_{\sfG, \beta}^0$ and $\langle \sigma_x \sigma_y \rangle_{\sfG, \beta}^+$. By passing to the infinite-volume measure, one can exploit the ergodic properties of the measures.
The following is proved for $\bbZ^d$ in $\cite{AizDumSid15}$, but the proof follows the same lines for any transitive graph.
\begin{theorem}
Let $\sfG$ be an infinite transitive graph. There exist two measures $\bbP ^{0}_{\sfG,\beta}$ and $\bbP ^{+}_{\sfG, \beta}$ on $\Omega_\sfG$ such that $\bbP ^{0}_{\sfG, \beta}$ is the weak limit of $\bbP_{\Lambda_n,\beta}$, and $\bbP ^{+}_{\sfG, \beta}$ is the weak limit of $\bbP_{\Lambda_n \cup \delta ,\beta}$, where $(\Lambda_n)_{n \geq 1}$ is a sequence of sets exhausting $\sfG$. Furthermore, both $\bbP ^{0}_{\sfG, \beta}$ and $\bbP ^{+}_{\sfG, \beta}$ are $\Gamma$-invariant and are ergodic with respect to the action of $\Gamma$. 
\end{theorem}
Let $\bbP ^{0,+}_{G, \beta}$ denote the law of the sum of two independent currents with laws $\bbP ^{0}_{\sfG, \beta}$ and $\bbP ^{+}_{\sfG, \beta}$.

For an infinite cluster of the current (more precisely infinite connected component of the multigraph associated with $\bfn$) define the number of its 
end to be the number of the ends of the connected component in the multigraph. Recall that an infinite graph $(\mathsf V, \mathsf E)$ has $k$ ends if 
the maximum, over all finite $\mathsf V' \subset \mathsf V$, of the number of infinite connected components of the induced subgraph of $\mathsf V 
\setminus \mathsf V'$ is equal to $k$. The following theorem, on the number of infinite clusters of the current and their number of ends, is proved using a 
Burton-Keane type argument in \cite{AizDumSid15}.

\begin{theorem} \label{thm:burtonkeane}
Let $\sfG$ be an amenable graph endowed with coupling constants satisfying \emph{(C1-C4)}. For the three measures $\bbP ^{0}_{G, \beta}$, $\bbP ^{+}_{G, \beta}$, and $\bbP ^{0,+}_{G, \beta}$  
the number of infinite clusters is at most $1$. Furthermore, if the infinite cluster exists, then its number of ends is at most $2$.
\end{theorem}


\paragraph{Notation}
From now on fix $\beta>0$ and drop it from the notation. We use $\sfB_n$ to denote the subgraph of $\sfG$ induced by vertices $\{x \in \sfG: d(0,x) \leq n\}$, where $d(.,.)$ is the graph distance. A sequence of vertices $(y_n )_{n \geq 1}$ converges to infinity if $d(0,y_n)$ tends to infinity.

We denote the value of the current $\bfn$ at $\{x,y\}$ by $\bfn_{xy}$. When $x$ and $y$ are not connected inside the set $\Lambda \subset \sfG$, we use the notation $x \notleftrightnn y$ in $\Lambda$. 

All the constants $c_i$ depend only on $\sfG$, $\{J_{xy}\}_{x,y \in \sfG}$, and $\beta$, and not on anything else. 
\section{Proof of Proposition \ref{prop:semimain}}

In the next lemmas we show that under the assumption that
\begin{equation}\label{eq:enemy}
\exists\, x, y \in \sfG, \,\,   J_{xy}>0: \, \langle \sigma_x \sigma_y \rangle ^+ > \langle \sigma_x \sigma_y \rangle ^0, \tag{$\mathfrak{A}$}
\end{equation}
 certain constraints are imposed on the geometry of infinite cluster of random current processes. 

For a current $\bfn$, let $\bfn_{[xy]}$ be the current which is equal to $\bfn$ on every $\{x',y'\}$ except at $\{x,y\}$ where it takes the value $0$ at $\{x,y\}$. 
For $x, y \in \sfG$, define the event $\calA_{x,y}$ to be the event consisted of $\bfn \in \Omega_\sfG$ with the following properties (See figure \ref{fig:eventA_inf}):
\begin{itemize}[noitemsep]
\item
$\bfn_{xy} =1$,
\item
$x \xleftrightarrow{\bfn_{[xy]}}  \infty$, $y \xleftrightarrow{\bfn_{[xy]}} \infty$, and $x \notleftrightxy y$. 
\end{itemize} 

\stepcounter{cst}
\begin{lemma} \label{lem:finsit}
Suppose \eqref{eq:enemy} holds. There exist $x,y \in \sfG$ and $c_{\thecst}>0$ such that
$$\bbP_{\sfG }^{+} \, \big[ \bfn \in \calA_{xy}  \big] > c_{\thecst},\text{ and } \, \, \bbP _{\sfG}^{0}\otimes\bbP _{\sfG}^{+} \, \big[ \bfn_1 + \bfn_2 \in \calA_{xy}\big] > c_{\thecst}.$$ 
\end{lemma}

\begin{proof}
Take $n$ large enough such that $\{x,y \} \subset \sfB_n$. Define the set $\calA_{xy}^f$ of currents $\bfn \in \Omega_{\sfB_n \cup \delta}$ with the following properties (See figure \ref{fig:eventA_f}): 
\begin{itemize}[noitemsep]
\item
$\bfn_{xy} =1$,
\item
$x \stackrel{\bfn_{[xy]} } {\longleftrightarrow} \delta$, $y \stackrel{\bfn_{[xy]}}{\longleftrightarrow} \delta$, and $ x \notleftrightxy y$ in $\sfB_n$. 
\end{itemize} 
It is sufficient to prove that there exists $c_{\thecst}>0$ such that for any $n$ sufficiently large,
\begin{equation} \label{eq:finitgeom}
\bbP_{\sfB_n \cup \delta} \, \big[ \bfn \in \calA_{xy}^f \big] > c_{\thecst} , \quad  
\bbP _{\sfB_n}\otimes\bbP _{\sfB_n \cup \delta} \, \big[ \bfn_1 + \bfn_2 \in \calA^f_{xy}\big] > c_{\thecst}.
\end{equation}
To prove \eqref{eq:finitgeom}, we start by writing correlations in terms of currents.
\begin{align}
\langle\sigma_x \sigma_y\rangle_{\sfB_n}^{+} - \langle \sigma_x \sigma_y\rangle_{\sfB_n}^{0}  &= \frac{\displaystyle\sum_{\bfn \in \Omega_{\sfB_n \cup \delta}, \: \partial \bfn = \{ x,y\}}  \omega (\bfn)}{\displaystyle\sum_{\bfn \in \Omega_{\sfB_n \cup \delta}, \: \partial \bfn = \emptyset}  \omega (\bfn)} -
\frac{\displaystyle\sum_{\bfn \in \Omega_{\sfB_n} , \: \partial \bfn = \{ x,y\}}  \omega (\bfn)}{\displaystyle\sum_{\bfn \in \Omega_{\sfB_n}, \: \partial \bfn =  \emptyset}  \omega (\bfn)} \nonumber \\[9pt]
&=  \frac{\displaystyle\sum_{\substack{\bfn_1 \in \Omega_{\sfB_n \cup \delta}, \: \partial \bfn_1 = \{ x,y\} \\ \bfn_2 \in \Omega_{\sfB_n}, \: \partial \bfn_2 = \emptyset}}  \omega(\bfn_1) \, \omega (\bfn_2) \big( 1 - {\bf1} \big[x\xleftrightarrow {\bfn_1+\bfn_2}y \text{ in $\sfB_n$} \big] \big)}
{\displaystyle\sum_{\substack{\bfn_1 \in \Omega_{\sfB_n \cup \delta}, \: \partial \bfn_1 = \emptyset \\ \bfn_2 \in \Omega_{\sfB_n}, \: \partial \bfn_2 = \emptyset}}  \omega(\bfn_1) \, \omega (\bfn_2) } \label{eq:a2} \\[9pt]
&\leq  \frac{\displaystyle\sum_{\substack{\bfn_1 \in \Omega_{\sfB_n \cup \delta}, \: \partial \bfn_1 = \{ x,y\} \\ \bfn_2 \in \Omega_{\sfB_n}, \: \partial \bfn_2 = \emptyset}}  \omega(\bfn_1) \, \omega (\bfn_2) {\bf1} \big[ \bfn'_1 + \bfn_2 \in \calA_{xy}^f \big] }
{\displaystyle\sum_{\substack{\bfn_1 \in \Omega_{\sfB_n \cup \delta}, \: \partial \bfn_1 = \emptyset \\ \bfn_2 \in \Omega_{\sfB_n}, \: \partial \bfn_2 = \emptyset}}  \omega(\bfn_1) \, \omega (\bfn_2) }, \label{eq:a3}
\end{align}
where $n'$ is the map defined on the current as
\begin{equation*}
\bfn'_{st} = \begin{cases}
\bfn_{xy} {+ 1} \quad \{s,t\}=\{x,y\}   \\
\bfn_{xy} {\phantom{ + 1}}  \quad \{s,t\}\neq   \{x,y\}   
\end{cases}.
\end{equation*}
Equality \eqref{eq:a2} is due to the switching lemma. Inequality \eqref{eq:a3} follows from two simple observations. First, since $x$ and $y$ are not connected inside $\sfB_n$,  $\bfn_1+ \bfn_2 $ has value $0$ on $\{x,y\}$, and hence, $(\bfn'_1 + \bfn_2)_{[xy]} = \bfn_1 + \bfn_2$. Second, since $\partial \bfn_1 = \{x,y\}$, $x$ and $y$ are connected in $\bfn_1$  in  $\sfB_n \cup \delta$. However, since  $x$ and $y$ are not connected inside $\sfB_n$, there exists a path connecting $x$ and $\delta$ and there is a path connecting $y$ and $\delta$ in $\bfn_1$.
The claim $\bbP _{\sfB_n}\otimes\bbP _{\sfB_n \cup \delta} \,\big[ \bfn_1 + \bfn_2 \in \calA^f_{xy}\big] > c_{\thecst}$ follows after several observations. First, the map $\bfn \rightarrow \bfn'$ is invertible. Second, 
\begin{equation}\label{eq:finenerg}
\omega(\bfn') = \frac{\beta J_{xy} }{n_{xy}+1} \, \omega(\bfn).
\end{equation} 
And third, since $\langle \sigma_x \sigma_y \rangle ^+ > \langle \sigma_x \sigma_y \rangle ^0$, there exists $c_2$, such that for all $n$ large enough 
$ c_2 \leq \langle\sigma_x \sigma_y\rangle_{\sfB_n}^{+} - \langle \sigma_x \sigma_y\rangle_{\sfB_n}^{0}.$ Altogether it implies
$$\bbP _{\sfB_n}\otimes\bbP _{\sfB_n \cup \delta} \,\big[ \bfn_1 + \bfn_2 \in \calA^f_{xy}\big] \geq \beta \,  J_{xy} \, c_2 =: c_\thecst. $$
To prove  $\bbP_{\sfB_n \cup \delta} \, \big[ \bfn \in \calA_{xy}^f \big] > c_{\thecst}$, note that 
\begin{align}
\langle\sigma_x \sigma_y\rangle_{\sfB_n}^{+} - \langle \sigma_x \sigma_y\rangle_{\sfB_n}^{0}  &\leq  \frac{\displaystyle\sum_{\substack{\bfn_1 \in \Omega_{\sfB_n \cup \delta}, \: \partial \bfn_1 = \{ x,y\} }}  \omega(\bfn_1) \, \left( 1 - {\bf1} \big[x\xleftrightarrow {\bfn_1}y \text{ in $\sfB_n$} \big] \right)}
{\displaystyle\sum_{\substack{\bfn_1 \in \Omega_{\sfB_n \cup \delta}, \: \partial \bfn_1 = \emptyset}}  \omega(\bfn_1) }. \nonumber \\[9pt]
&\leq  \frac{\displaystyle\sum_{\substack{\bfn_1 \in \Omega_{\sfB_n \cup \delta}, \: \partial \bfn_1 = \{ x,y\}}}  \omega(\bfn_1)\, {\bf1} \big[ \bfn'_1 \in \calA_{xy}^f] }
{\displaystyle\sum_{\substack{\bfn_1 \in \Omega_{\sfB_n \cup \delta}, \: \partial \bfn_1 = \emptyset}}  \omega(\bfn_1) }. \label{eq:a5}
\end{align}
This is true since
${\bf1} \big[x\xleftrightarrow {\bfn_1}y \text{ in $\sfB_n$} \big] \leq {\bf1} \big[x\xleftrightarrow {\bfn_1+\bfn_2}y \text{ in $\sfB_n$} \big].$
The same reasoning as the previous paragraph concludes that \eqref{eq:a5} implies 
$\bbP_{\sfB_n \cup \delta} \, \big[ \bfn \in \calA_{xy}^f \big] >  c_1.$
\end{proof}
\begin{figure}[htp]
\centering
\hfill
\begin{minipage}[t]{.44\linewidth}
\centering
  \includegraphics[width=.9\linewidth]{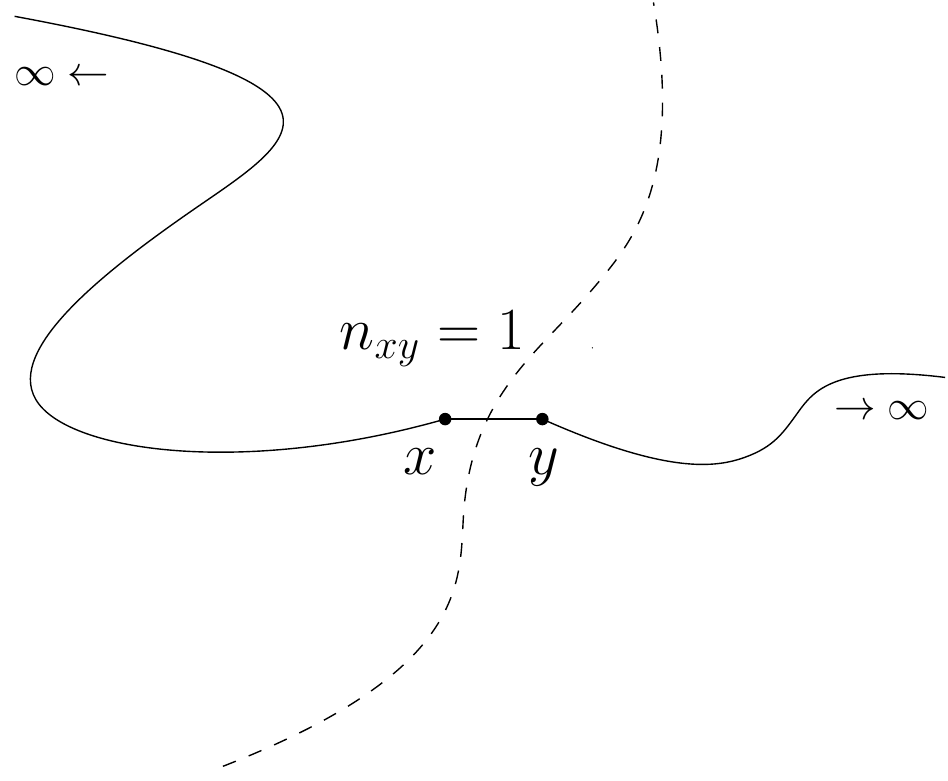}
  \caption{A diagrammatic representation of the event $\calA_{xy}$. The dashed line schematically indicates that there is no connection between $x$ and $y$ when the the value of the current on $\{x,y\}$ is set to zero.}
  \label{fig:eventA_inf}
\end{minipage}
\hfill
\hfill
\begin{minipage}[t]{.44\linewidth}
\centering
  \includegraphics[width=0.9\linewidth]{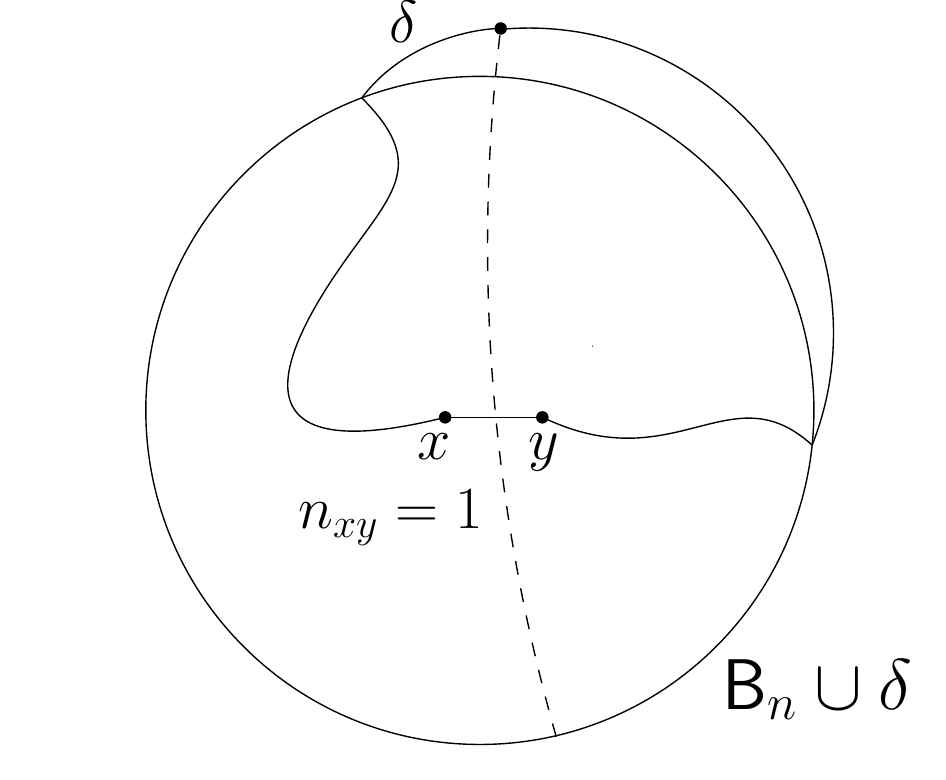}
  \caption{A diagrammatic representation of the event $\calA_{xy}^f$. The dashed line schematically indicates that there is no connection between $x$ and $y$ inside $\sfB_n$ when the value of the current on $\{x,y\}$ is set to zero.}
  \label{fig:eventA_f}
\end{minipage}
\hfill
\end{figure}

For a current $\bfn\in \Omega_\sfG$, define
$$\mathsf{U}(\bfn) := \big \lbrace \{x,y\} \in \calP_2 (\mathsf{G}) : \bfn \in \mathcal{A}_{xy}\big \rbrace.$$
For $x,y \in G$ and a current $\bfn$, define the flow $\mathsf{f}(x,y, \bfn)$ between $x$ and $y$ in $\bfn$ to be
the maximal number of edge-disjoint paths between $x$ and $y$ in the multigraph associated with $\bfn$. Define the function $\mathsf{f}(x,\infty, \bfn)$ to be the maximal number of edge-disjoint one-sided infinite path starting from $x$ in the multigraph associated with $\bfn$. 

We state the following lemma in generality instead of for the specific Ising random currents measures.
\begin{lemma} \label{lem:flow}
Let ${\bf P}$ be a measure on $\Omega_\sfG$ which is ergodic with respect to the action of $\Gamma$ satisfying
$$ {\bf P} [\mathsf U(\bfn) \neq \emptyset] >0.$$
Furthermore, assume that ${\bf P}$ almost surely there exists a unique infinite cluster and the number of its ends is at most $2$.
Then, for every $x \in G$
\begin{itemize}
\item[(i)]
$ \displaystyle\lim_{y \rightarrow \infty} \, {\bf P} \, \big[\mathsf f (x,y, \bfn) \geq 2 \big] = 0,$
\item[(ii)]
${\bf P}\, \big[ \mathsf f (x,\infty, \bfn) >2 \big] = 0.$
\end{itemize}
\end{lemma}

\begin{proof}
First we prove the following claim.
\begin{claim}
For ${\bf P}$ almost every $\bfn$, for any infinite path $\{ v_j\}_{j \geq 1}$ in $\bfn$ there exists $i$ such that $\{v_i, v_{i+1}\} \in U(\bfn)$.
\end{claim}

\noindent{\em Proof of Claim 1.}
Suppose that with positive probability, there exists a configuration $\bfn$ 
such that there exists a path $\{v_i\}_{i \geq 1}$  in $\bfn$
such that for any $i$, $\{v_i,v_{i+1} \} \notin \mathsf{U}(\bfn)$.  
Fix such a configuration $\bfn$ and denote the multigraph associated with $\bfn$ by ${\bf M}$. 

Note that because the law of $\bfn$ is ergodic, $\mathsf{U}(\bfn)$ is infinite. Uniqueness of the infinite cluster implies that for any $\{u_0,u_1\} \in \mathsf{U}(\bfn)$
 and for any $i$, there exists a path from $\{u_0, u_1\}$ to  $v_i$ in $\bfn$. 
Hence, by relabeling the vertices we can assume there exists a self-avoiding infinite path $u_0$, $u_1$, $v_1$, $v_2$, $\dots$, $v_n$, $\dots$ in ${\bf M}$ such that $\{u_0,u_1\} \in \mathsf{U}(\bfn  )$, and for any $i$, $\{v_i,v_{i+1} \} \notin \mathsf{U}(\bfn)$.  

Let us show that in this situation there exist two edge-disjoint one-sided infinite paths starting from $u_1$ in ${\bf M}$, none of them passing through the edge $\{u_0,u_1\}$. 
First, at least one such path exists: since $u_1$, $v_1$, $v_2$, $\dots$, $v_n$, $\dots$ is self-avoiding, it does not cross the edge $\{u_0,u_1\}$.  
Suppose there exists only one such path, then Menger's maxflow-mincut theorem implies existence of $\{v_k,v_{k+1}\}$ such that for $j\geq k$, any path from $u_1$ to $v_{j}$ in ${\bf M}$ passes through $\{v_k,v_{k+1}\}$. 
But note that $\{v_k,v_{k+1}\} \notin \mathsf{U}(\bfn)$, and hence there exists a path connecting $v_ {k+1}$ to $u_{1}$ not 
crossing through the edge $\{v_k,v_{k+1}\}$.
Therefore there is a path from $u_1$ to $v_{k+1}$ in ${\bf M}$ not using $\{v_k,v_{k+1}\}$, which is a contradiction. 
Hence, the assumption was false and there exist two edge-disjoint one-sided infinite paths starting from $u_1$ which do not pass through the edge $\{u_0,u_1\}$. 

Ergodicity of the measure ${\bf P}$ implies that almost surely there exist infinitely many $\{u_0, u_1\}$ such that $\{u_0, u_1\} \in 
\mathsf{U}(\bfn)$ and there exist two edge-disjoint one-sided infinite paths from $u_1$ in ${\bf M}$ which do not cross the edge 
$\{u_0,u_1\}$. Take $\{u_0, u_1\}$, $\{u'_0, u'_1\}$, $\{u''_0, u''_1\}$  three of such pairs. Uniqueness of the infinite cluster implies that they are all 
connected to each other. Without loss of generality, we can assume that $\{u_0, u_1\}$ is connected to $u'_0$ and $u''_0$ without using any of the edges 
$\{u_0, u_1\}$, $\{u'_0, u'_1\}$, and $\{u''_0, u''_1\}$. Since there exist three edge-disjoint infinite one-sided paths in ${\bf M}$ starting from $u_1$, 
one of these paths does not pass through any of the edges $\{u'_0, u'_1\}$ and $\{u''_0, u''_1\}$. Call this path $p = u_1, u_2, \dots$, and also note that
there exist two other infinite paths $p' = u'_1, u'_2, \dots$ and $p'' = u''_1, u''_2, \dots$ both of them disjoint from both of $\{u'_0, u'_1\}$ and $\{u''_0, u''_1\}$.

Now, in the induced subgraph of $\mathsf{G} \setminus \{u'_0, u'_1, u''_0, u''_1\}$ in ${\bf M}$, the paths $p$, $p'$, and $p''$ belong to three different connected components. This implies that the number of ends of the infinite cluster of $\bfn$ is at least $3$ which contradicts the assumption on the measure ${\bf P}$.

\hfill $\diamond$ \vspace{6pt}\\
We prove item $(i)$. For a current $\bfn$, define
\begin{equation*}
\bfn'_{x,y} =
\begin{cases}
\bfn_{xy} &\{x,y\} \notin \mathsf U(\bfn)\\
0 \quad &\{x,y\} \in \mathsf U(\bfn)
\end{cases}.
\end{equation*}
Claim 1 implies that ${\bf P}$ almost surely $\bfn'$ has no infinite cluster. Hence for any $x$, $$\lim_{y \rightarrow \infty} {\bf1} \big[x \stackrel{\bfn'}{\longleftrightarrow } y \big] = 0.$$ 
If $\mathsf f(x,y, \bfn) \geq 2$ then $x$ is connected to $y$ in $\bfn'$, hence almost surely
$$\lim_{y \rightarrow \infty} {\bf1} \big[ \mathsf f (x,y, \bfn) \geq 2 \big] = 0.$$
The convergence of the expectation follows from the bounded convergence theorem.

We turn to the proof of item $(ii)$. 
Take a configuration $\bfn$ such that there exist three one-sided infinite paths starting from $x$. 
Claim 1 states that each of them has an element of $\sf U(\bfn)$ and from this and the definition of the $\sf U(\bfn)$ we infer that the number of ends of the infinite cluster is at least 3, which contradicts Theorem \ref{thm:burtonkeane}.
\end{proof}
  \begin{figure}[htp]
    \centering
    \includegraphics[width=.58\linewidth]{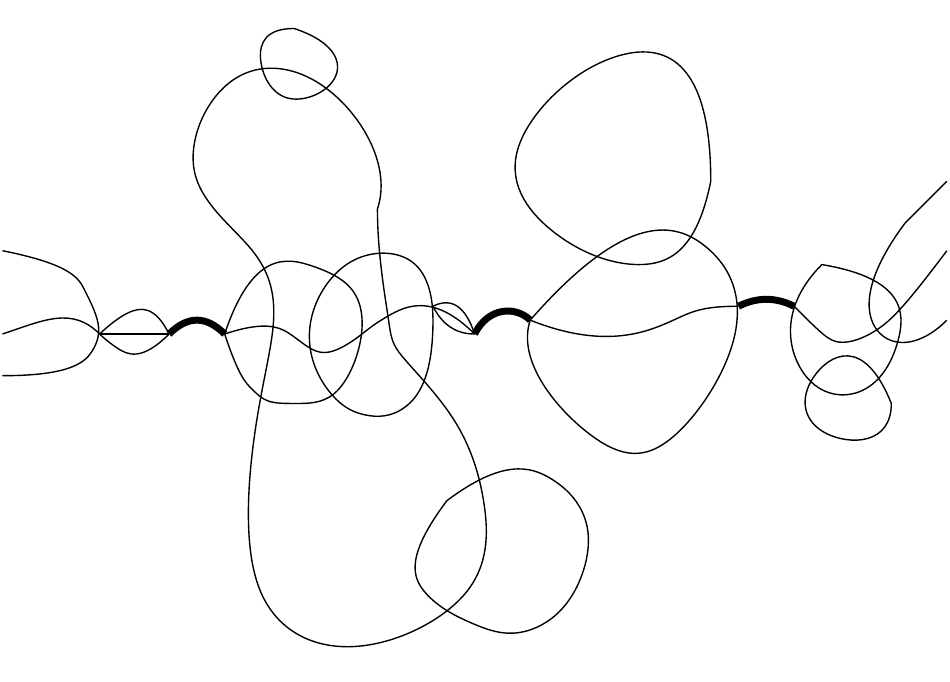}
    \caption{The geometry of the infinite cluster of $\bbP ^{0,+}_{G}$, assuming \eqref{eq:enemy}, based on Lemma \ref{lem:flow}. A string of finite loop-clusters are connected to each other by elements of $\mathsf U (\bfn_1+ \bfn_2)$. Elements of $\mathsf U (\bfn_1+ \bfn_2)$ are depicted in bold.}
    \label{crossing:fig:constr3}
  \end{figure}
\stepcounter{cst} \stepcounter{cst} 
Lemma  \ref{lem:finsit} stated that for a current on $\sfB_n \cup \delta$, having sources at $\{x,y\}$, the probability that $x$ and $y$ are only connected through $\delta$ is positive. However, the next lemma states that with positive probability $x$ and $y$ are connected to each other without using $\delta$.
\begin{lemma} \label{lem:n1con}
Assume \eqref{eq:enemy} is correct. Then, there exists $c_{\thecst}>0$ such that for all $x\in \mathsf G$ there exist infinitely many $y \in \sf G$, such that 
\begin{equation} \label{eq:jiljul}
\langle \sigma_x \sigma_y \rangle^+_{\sfB_n}  \bbP_{\mathsf{B}_n \cup \delta}^{\{x,y\}}\, \big[x \xleftrightarrow{\bfn} y \text{ in }  \mathsf{B}_m \big] > c_{\thecst},
\end{equation}
for $m=m(x,y)$ and all $n$ sufficiently large.
\end{lemma}
\begin{proof}
We demonstrate for any $d_0>0$ there exists $y$ satisfying \eqref{eq:jiljul} and $d(x,y) > d_0$.
First note that there exists $c_4 >0$ such that for any $x \in \sfG$, there exists $y,x', y' \in \sfG$ with $d (\{x,x'\}, \{y,y'\} ) > d_0$ and 
\begin{equation}\label{eq:twopointu}
\bbP^+_\sfG \, \big[  \{x,x'\}, \{y, y'\}  \in \mathsf{U} (\bfn) \big] \geq c_4.
\end{equation}
This is due to Lemma \ref{lem:finsit} and ergodicity of the measure $\bbP^+_\sfG$. More precisely, there exists $R$ sufficiently large such that for any $x \in \sfG$, 
the probability that $\mathsf U(n)$ intersects $\sfB_R$ is at least $3/4$.
Fix $d_0$, and take two vertices $s$ and $t$ such that $d(s,t) \geq d_0+ 2R$. Now
$$\bbP^+_\sfG \, \big[ \mathsf{B}_R(s) \cap \mathsf{U}(\bfn)  \neq \emptyset, \mathsf{B}_R(t) \cap \mathsf{U}(\bfn_1)  \neq \emptyset \big] \geq 1/2,$$
and this together with $\Gamma$-invariance of the measure implies \eqref{eq:twopointu} with $c_4 := 1/(4|\mathsf B_R|^4)$.
\medskip

For a current $\bfn$ with  $\bfn_{xx'} >0$ and $\bfn_{yy'}>0$, let $\bfn'$ denote the current satisfying
\begin{equation*}
\bfn'_{st} =
\begin{cases}
\bfn_{st} - 1\quad  &\text{if }\{s,t\} =\{x,x'\} \text{ or } \{s,t\} =\{y,y'\},\\
\bfn_{st}\quad &\text{otherwise.}
\end{cases}.
\end{equation*}
Define 
$$\calB_m(x,x',y,y'):=  \{ \bfn \in \Omega_{\sfB_n \cup \delta}: x\stackrel{\bfn}{\longleftrightarrow} y \text{ in } \mathsf B_m , x \notleftrightnn x', y \notleftrightnn y'\}.$$ 
Let $\bfn$ be a current such that 
$\{x,x'\}, \{y, y'\} \in  \mathsf U (\bfn)$. 
The definition of $\mathsf{U}(\bfn)$ and the uniqueness of the infinite cluster guarantee that $\bfn' \in \calB_m(x,x',y,y')$ occurs for some $m$ (perhaps after exchanging $x$ and $x'$, and $y$ and $y'$). Hence \eqref{eq:twopointu} implies that there exists $m = m(x,y,x',y')$ large enough  such that
$$\bbP^+_\sfG \, \big[\bfn' \in  \calB_m(x,x',y,y'), \bfn_{xx'} =1, \bfn_{yy'}=1 \big] > c_4/8 .$$
This in turn gives that for all $n$ sufficiently large,
$$\bbP^+_{\sfB_n \cup \delta} \, \big[ \bfn' \in \calB_m(x,x',y,y'), \bfn_{xx'} =1, \bfn_{yy'}=1  \big] > c_4/8.$$
Note that $\bfn \longrightarrow \bfn'$ is an invertible map from the set of currents without sources to the set of currents with sources $\{x,x',y,y'\}$, and
$$\omega(\bfn') = \frac{\bfn_{xx'} \bfn_{yy'}}{\beta^2 J_{xx'} J_{yy'}} \, \omega(\bfn).$$
This implies the existence of $c_3>0$ such that for all $n$ sufficiently large
\begin{align} \label{eq:fso}
\displaystyle\sum_{\substack{\bfn \in \Omega_{\sfB_n \cup \delta}, \: \partial \bfn =\{ x,x',y',y'\} }}  \omega(\bfn) \,{\bf1} \big[\bfn \in \calB_m(x,y,x',y')\big]   \geq c_3 {\displaystyle\sum_{\bfn \in \Omega_{\sfB_n \cup \delta}, \: \partial \bfn = \emptyset}  \omega (\bfn) }.
\end{align}

In order to conclude the lemma from \eqref{eq:fso}, we need to extract information for configurations with two sources $\{x,y\}$ from configurations with four sources $\{ x,x',y,y'\}$. 

For the current configuration $\bfn$, define
$\mathsf{C}_{x} (\bfn)$ 
to be the the connected component of $x$ in the multigraph of $\bfn$. 

Call a multigraph $C$ \emph{admissible} if and only if $\{x,y \} \subset \mathsf V(C)$ and $\{x',y' \} \cap \mathsf V(C) = \emptyset$. 
For any current $\bfn$, no matter what are its sources, 
$\bfn \in \calB_m(x,y,x',y')$ if and only if $\mathsf{C}_{x}(\bfn)$ is admissible. Hence, for any $A \subset \mathsf B_n \cup \delta$,
\begin{align} \label{eq:boo1}
\displaystyle\sum_{\substack{\bfn \in \Omega_{\sfB_n \cup \delta}\\ \partial \bfn =A }}  \omega(\bfn) \,{\bf1} \big[\bfn \in \calB_m(x,y,x',y')\big]  =\displaystyle\sum_{\substack{C \text{ admissible }}} \displaystyle\sum_{\substack{\bfn \in \Omega_{\sfB_n \cup \delta} \\ \partial \bfn = A}}  \omega(\bfn) \,  {\bf1} \big[\mathsf C_{x} (\bfn) = C \big]. 
\end{align}

For a multigraph $C$  and a current $\bfn$, $\mathsf C_x(\bfn)= C$ if and only if $\bfn$ can be uniquely decomposed into $\bfn = \bfn_1+ \bfn_2$ such that the multigraph associated with $\bfn_1$ equals to $\mathsf C_x(\bfn)$, and $\bfn_2 \in \Omega_{(\sfB_n \cup \delta) \setminus \mathsf V(C)}$. 
Note that $\omega(\bfn) = \omega(\bfn_1) \, \omega(\bfn_2)$.
Furthermore if  $\partial n = \{x,x',y,y'\}$ and $C$ is admissible, then
$\partial \bfn_1 = \{x, y\}$, and $\partial \bfn_2 = \{ x',y'\}$. 
Altogether it implies for an admissible $C$, 
\begin{align} \label{eq:boo2}
 \displaystyle\sum_{\substack{\bfn \in \Omega_{\sfB_n \cup \delta} \\ \partial \bfn = \{x,x',y,y'\}}}  \omega(\bfn) \,  {\bf1} \big[\mathsf C_{x} (\bfn) = C \big]
=\omega(\bfn (C))   \displaystyle\sum_{\substack{\bfn_2 \in \Omega_{(\sfB_n \cup \delta) \setminus \mathsf V(C)}  \\ \partial \bfn_2 = \{x',y'\}}} \omega(\bfn_2),
\end{align}
where for a multigraph $C$,  $\bfn(C)$ denotes the current which its multigraph is equal to $C$. In the same manner, for an admissible $C$
\begin{align} \label{eq:boo3}
 \displaystyle\sum_{\substack{\bfn \in \Omega_{\sfB_n \cup \delta} \\ \partial \bfn = \{x,y\}}}  \omega(\bfn) \,  {\bf1} \big[\mathsf C_{x} (\bfn) = C \big]
=\omega(\bfn (C))   \displaystyle\sum_{\substack{\bfn_2 \in \Omega_{(\sfB_n \cup \delta) \setminus \mathsf V(C)}  \\ \partial \bfn_2 = \emptyset}} \omega(\bfn_2).
\end{align}
Applying \eqref{eq:iscor} implies
\begin{align*}
\displaystyle\sum_{\substack{\bfn_2 \in \Omega_{(\sfB_n \cup \delta) \setminus \mathsf V(C)}  \\ \partial \bfn_2 = \{x',y'\}}} \omega(\bfn_2) 
=   \langle \sigma_{x'} \sigma_{y'} \rangle_{(\mathsf B_n \cup \delta) \setminus \mathsf V(C)} \displaystyle\sum_{\substack{\bfn_2 \in(\mathsf B_n \cup \delta) \setminus \mathsf V(C)  \\ \partial \bfn_2 = \emptyset}} \omega(\bfn_2)
\leq  \displaystyle\sum_{\substack{\bfn_2 \in(\mathsf B_n \cup \delta) \setminus \mathsf V(C)  \\ \partial \bfn_2 = \emptyset}} \omega(\bfn_2), 
\end{align*}
which together with \eqref{eq:boo1}, \eqref{eq:boo2}, and \eqref{eq:boo3} leads to
\begin{align*}
\displaystyle\sum_{\substack{\bfn \in \Omega_{\sfB_n \cup \delta}\\ \partial \bfn =\{ x,x',y',y'\} }}  \omega(\bfn) \,{\bf1} \big[\bfn \in \calB_m(x,x',y,y')\big]  \leq& \displaystyle\sum_{\substack{\bfn \in \Omega_{\sfB_n \cup \delta} \\ \partial \bfn = \{x,y\}}}  \omega(\bfn) \,  {\bf1} \big[\bfn \in \calB_m(x,x',y,y')\big]\\
\leq& \displaystyle\sum_{\substack{\bfn \in \Omega_{\sfB_n \cup \delta} \\ \partial \bfn = \{x,y\}}}  \omega(\bfn) \,  {\bf1} \big[x\xleftrightarrow{\bfn} y \text{ in } \sfB_m  \big].
\end{align*}
Combining this with Equation \eqref{eq:fso} and then applying \eqref{eq:ispcor} conclude the proof of the Lemma.
\end{proof}
\stepcounter{cst} \stepcounter{cst}

We are now ready to prove Proposition \ref{prop:semimain}.
\begin{proof}[Proof of Proposition \ref{prop:semimain}]
Assume the proposition is wrong, and hence \eqref{eq:enemy} holds true. 
We prove the proposition by considering two different cases and reaching a contradiction in each case by different means. 

\paragraph{First Case.} $\displaystyle\lim_{y \rightarrow \infty} \langle \sigma_x \sigma_y \rangle^{0} = 0$ or $ \bbP^{0}_\sfG \big[x \stackrel{\bfn}{\longleftrightarrow} \infty \big]>0.$
\vspace{0.6cm}

\noindent Assume $\lim_{y \rightarrow \infty} \langle \sigma_x \sigma_y \rangle^{0} = 0$. Lemma \ref{lem:switch} implies
\begin{equation} \label{eq:aux2}
 \lim_{y \rightarrow \infty}\bbP _{G}^{0}\otimes\bbP _{G}^{+} \big[ x \xleftrightarrow{\bfn_1+ \bfn_2} y\big] =  \lim_{y \rightarrow \infty} \, \langle \sigma_x \sigma_y \rangle ^+  \langle \sigma_x \sigma_y \rangle ^0 \leq  \lim_{y \rightarrow\infty}  \, \langle \sigma_x \sigma_y \rangle ^0 = 0.
 \end{equation}
Suppose $\bbP _{G}^{0}\otimes\bbP _{G}^{+} \big[ x \xleftrightarrow{\bfn_1+ \bfn_2} \infty\big] >0$, then for $R$ sufficiently large and any $x, y \in \sfG$
$$\bbP _{G}^{0}\otimes\bbP _{G}^{+} \, \big[ \mathsf{B}_R(x) \xleftrightarrow{\bfn_1+ \bfn_2} \infty, \mathsf{B}_R(y)  \xleftrightarrow{\bfn_1+ \bfn_2} \infty \big] \geq 1/2.$$
This implies that for any $d_0>0$ there exist $x,y \in \sfG$, with $d(x,y) > d_0$, such that $\bbP _{G}^{0}\otimes\bbP _{G}^{+} \, \big[x \xleftrightarrow{\bfn_1+ \bfn_2} y\big] \geq 1 / 2|\sfB_R|^2$ which contradicts \eqref{eq:aux2}.
Hence,
\begin{equation} \label{eq:sumnotpercolate}
\bbP _{G}^{0}\otimes\bbP _{G}^{+} \big[ x \xleftrightarrow{\bfn_1+ \bfn_2} \infty\big] = 0.
\end{equation}
However, assuming \eqref{eq:enemy}, there exist $x',y' \in \sfG$ such that 
$\bbP _{G}^{0}\otimes\bbP _{G}^{+} [\calA_{x'y'}] >0. $
On the event $\calA_{x'y'}$, $x'$ belongs to an infinite cluster of $\bfn_1+\bfn_2$, which implies
 \begin{equation*}  
\bbP _{G}^{0}\otimes\bbP _{G}^{+} \big[ x' \xleftrightarrow{\bfn_1+\bfn_2} \infty] > 0.
\end{equation*}
This is contradicting \eqref{eq:sumnotpercolate}.
\begin{remark}
Note that here, since the $\langle . \rangle_\beta ^0$ is in the disordered phase, we gather that the set of Gibbs state is a singleton. 
The fact that $\lim_{y \rightarrow \infty} \langle \sigma_x \sigma_y \rangle^{0} = 0$ implies uniqueness of the Gibbs state is proved in \cite{AizDumSid15}, and we presented the proof for the sake of completeness.
\end{remark}
\vspace{0.4cm}
Assume $ \bbP^{0}_\sfG \big[x \stackrel{\bfn}{\longleftrightarrow} \infty \big]>0.$
Lemma \ref{lem:finsit} implies that assuming \eqref{eq:enemy},
$\bbP^+_\sfG \big[ \mathsf f(x,\infty, \bfn) \geq 2 \big] \geq c_1,$ 
and hence
$$\bbP^+_\sfG \otimes \bbP^0_\sfG \, \big[  \mathsf f(x,\infty, \bfn_1+ \bfn_2) \geq 3 \big] \geq c_1 \, \bbP^{0}_\sfG \big[x \stackrel{\bfn}{\longleftrightarrow} \infty\big] >0. $$
This contradicts the second item of Lemma \ref{lem:flow}.

\paragraph{Second Case.} $\displaystyle\lim_{y \rightarrow \infty} \langle \sigma_x \sigma_y \rangle^{0} > 0$ and $ \bbP^{0}_\sfG \big[x \xleftrightarrow{\bfn} \infty \big]=0.$

\vspace{0.6cm}

\noindent Let us begin by showing why it suffices to prove existence of $c_{\thecst}>0$ such that for any $x, y \in \sfG$, there exists $m=m(x,y)$ such that for all $n$ sufficiently large
\begin{equation} \label{eq:eqjin}
\langle \sigma_x \sigma_y\rangle^0_{\mathsf B_n} \, \bbP_{\mathsf{B}_n }^{\,x,y}\, \big[x \stackrel{\bfn}{\longleftrightarrow} y \text{ in }  \mathsf{B}_m \big] > c_{\thecst}.
\end{equation}

Assume \eqref{eq:eqjin} and take $d_0 \geq 1$ arbitrarily large. Lemma \ref{lem:n1con} implies the existence of $y$ such that $d(x,y) >d_0$ and the existence of $m$ such that for all $n$ sufficiently large 
$$ \langle \sigma_x \sigma_y\rangle_{\mathsf B_n}^+ \, \, \, \bbP_{\mathsf{B}_n \cup \delta}^{\,x,y}\, \big[x \leftrightarrow y \text{ in }  \mathsf{B}_m \big] > c_3.$$
Hence for $n$ sufficiently large enough 
\begin{equation}  
\frac{\displaystyle\sum_{\substack{\bfn_1 \in \Omega_{\sfB_n \cup \delta}, \: \partial \bfn_1 = \{ x,y\} \\ \bfn_2 \in \Omega_{\sfB_n}, \: \partial \bfn_2 = \{ x,y\}}}  \omega(\bfn_1) \, \omega (\bfn_2) {\bf1} \big[x \stackrel{\bfn_1} {\longleftrightarrow} y \text{ in } \mathsf B_m, \, \, x \stackrel{\bfn_2} {\longleftrightarrow} y \text{ in } \mathsf{B}_m  \big] \,} 
{\displaystyle\sum_{\substack{\bfn_1 \in \Omega_{\sfB_n \cup \delta}, \: \partial \bfn_1 = \emptyset \\ \bfn_2 \in \Omega_{\sfB_n}, \: \partial \bfn_2 = \emptyset}}  \omega(\bfn_1) \, \omega (\bfn_2) } \geq c_3 c_{\thecst}.
\end{equation} 
Note that if there exist a path in $\bfn_1$ in $\mathsf B_m$ and a path in $\bfn_2$ in $\mathsf B_m$, then there are at least 2 disjoint paths between $x$ and $y$ in the multigraph of $\bfn_1+\bfn_2$.
The switching lemma thus implies that
\begin{equation*}  
\frac{\, \displaystyle\sum_{\substack{\bfn_1 \in \Omega_{\sfB_n \cup \delta}, \: \partial \bfn_1 = \emptyset \\ \bfn_2 \in \Omega_{\sfB_n}, \: \partial \bfn_2 = \emptyset}}  \omega(\bfn_1) \, \omega (\bfn_2) \, {\bf1}\big[\mathsf{f}(x,y,\bfn_1+\bfn_2, \mathsf B_m) \geq 2 \big]  \,} 
{\displaystyle\sum_{\substack{\bfn_1 \in \Omega_{\sfB_n \cup \delta}, \: \partial \bfn_1 = \emptyset \\ \bfn_2 \in \Omega_{\sfB_n}, \: \partial \bfn_2 = \emptyset}}  \omega(\bfn_1) \, \omega (\bfn_2) } \geq c_3 c_{\thecst},
\end{equation*}
where $\mathsf{f}(x,y,\bfn_1+\bfn_2, \mathsf{B}_m)$ is the maximal number of edge-disjoint paths in $\bfn_1+\bfn_2$ which use only vertices of $\mathsf B_m$. Equivalently,  $\bbP _{\sfB_n}\otimes\bbP _{\sfB_n \cup \delta} \big[\mathsf{f}(x,y,\bfn_1+\bfn_2, \mathsf B_m) \geq 2\big] \geq c_3 c_{\thecst}.$
Taking the limit as $n$ tends to infinity implies that
\begin{equation}\label{eq:aux8}
\bbP^{0} _{G}\otimes\bbP^{+} _{G} \, \big[ \mathsf{f} (x,y, \bfn_1 + \bfn_2, \mathsf{B}_m) \geq 2 \big] \geq c_3 c_{\thecst}.
\end{equation}
Since $\mathsf{f} (x,y, \bfn_1 + \bfn_2) \geq \mathsf{f} (x,y, \bfn_1 + \bfn_2, \mathsf B_m) $, \eqref{eq:aux8} contradicts the first item of  Lemma \ref{lem:flow}.
\stepcounter{cst}
\medbreak
Now, we prove that if $\displaystyle\lim_{y \rightarrow \infty} \langle \sigma_x \sigma_y \rangle^{0} > 0$ and \eqref{eq:eqjin} is not true then
\begin{align} \label{eq:n0perco}
 \bbP^{0}_\sfG \big[x \stackrel{\bfn}{\longleftrightarrow} \infty \big]>0.
 \end{align}
This proves \eqref{eq:eqjin} in this case.

If $\displaystyle\lim_{y \rightarrow \infty} \langle \sigma_x \sigma_y \rangle^0 >0$ and \eqref{eq:eqjin} is not true, then there exist $c_{\thecst}>0$ and $x,y \in \sfG$ such that for any $m \geq 0$,
\begin{align} \label{eq:3rdcase} 
\langle \sigma_x \sigma_y\rangle^0_{\mathsf B_n} \, \bbP_{\mathsf{B}_n }^{\,x,y}\, \big[x \stackrel{\bfn_1}{\longleftrightarrow}  \partial \mathsf{B}_m \big] > c_{\thecst}
\end{align}
for all $n$ sufficiently large.
This is due to the fact that $\partial \bfn = \{x, y\}$, there exist a path between $x$ and $y$, and if this path is not contained in $\mathsf B_m$, then there exists a path in $\bfn$ from $x$ to $\partial \mathsf{B}_m$. 

In order to prove \eqref{eq:n0perco} from \eqref{eq:3rdcase}, we need to extract information for currents without sources from currents with sources $\{ x, y\}$.  
Take $n$ large enough and introduce a mapping
$$\mathsf T: \big \lbrace \partial \bfn = \{x,y\}, x \stackrel{\bfn}{\longleftrightarrow} \partial \mathsf B_m \big \rbrace \rightarrow \{0,1\} ^{ \lbrace \partial \bfn  = \emptyset, x \stackrel{\bfn}{\longleftrightarrow} \partial \mathsf B_m \rbrace}.$$ 
To define the map $\sf T$, fix a sequence of vertices $x=v_0, v_1, \dots, v_k=y$  in $\mathsf B_m$ such that 
$J_{v_i v_{i+1}}>0$ for $0 \leq i \leq k-1$. For any current $\bfn$, the current $\bfn' \in \mathsf{T}(\bfn)$ if and only if $\bfn'_{v_iv_{i
+1}}$ is a nonzero integer with different parity from $\bfn_{v_i v_{i+1}}$ and $\bfn'_{st} = \bfn_{st}$ for any other $s$ and $t$. Note that if $
\partial \bfn = \{x,y\} $ then for any $\bfn' \in \mathsf T(n)$, $\partial \bfn' = \emptyset$. Also if $x \stackrel{\bfn}{\longleftrightarrow} \partial \mathsf B_m$, 
then for  $\bfn' \in \mathsf T(n)$, $ x \xleftrightarrow{\bfn'} \partial \mathsf B_m$.  

For  $\bfn', \bfn'' \in \big \lbrace \partial \bfn  = \emptyset, x \xleftrightarrow{\bfn} \partial \mathsf B_m \big \rbrace$, define $\bfn' \sim \bfn''$ if and only if there exists $\bfn$ such that $\bfn', \bfn'' \in \mathsf T(\bfn)$. For the current $\bfn'$
$$ \sum_{\substack{\bfn \, \in \mathsf T^{-1}(\bfn')}} \omega(\bfn) \leq K \, \displaystyle \sum_{\bfn'' \sim \bfn'} \omega(\bfn''),$$ 
where $$K = \prod_{0 \leq i \leq k-1} \max \bigg\lbrace \frac{\cosh \beta J_{v_iv_{i+1}} }{\sinh \beta J_{v_iv_{i+1}}}, \frac{\sinh \beta J_{v_iv_{i+1}}}{\cosh \beta J_{v_iv_{i+1}} -1}\bigg \rbrace.$$
This implies
\begin{align*}
\displaystyle \sum_{\partial \bfn = \{ x,y\}} \omega(\bfn) \, {\bf 1} \big[ x \stackrel{\bfn}{\longleftrightarrow} \partial \mathsf {B}_m\big]\leq
K \displaystyle \sum_{\partial \bfn = \emptyset}  \omega(\bfn) \, {\bf 1} \big[ x \stackrel{\bfn}{\longleftrightarrow} \partial \mathsf {B}_m\big].
\end{align*}
This together with \eqref{eq:3rdcase} implies that for any $m>0$,  
\begin{align*}
&\frac{\,\displaystyle\sum_{\substack{\bfn \in \Omega_{\mathsf B_n }, \, \partial \bfn = \emptyset}}  \omega(\bfn) \,  {\bf1} \big[x \stackrel{\bfn}{\longleftrightarrow} \partial \mathsf B_m] \,}{\displaystyle\sum_{\bfn \in \Omega_{\mathsf B_n }, \: \partial \bfn = \emptyset}  \omega (\bfn) } \geq \frac{c_{\thecst}}{K}
\end{align*}
for $n$ sufficiently large, which readily implies \eqref{eq:n0perco}. 
\end{proof}

%
%

\section{Further Questions}
We finish this article by presenting a few open questions. We proved Theorem \ref{thm:main} for amenable graphs. It is noteworthy that there are examples of non-amenable graphs for which Theorem \ref{thm:main} is not valid. For instance, on the regular tree, for any $\beta > \beta_c$,  $ \langle . \rangle^0_\beta$ is not equal to $\tfrac 1 2 \langle . \rangle^+_\beta + \tfrac 1 2 \langle . \rangle ^-_\beta $. 
This leads to the following question:

\begin{question} \label{q:pcputype}
Let $\sfG$ be a non-amenable transitive graph. Does there exist $\beta$  such that $ \langle . \rangle^0_\beta \neq \tfrac 1 2 \langle . \rangle^+_\beta + \tfrac 1 2 \langle . \rangle ^-_\beta $?
\end{question}

On the other hand, even on non-amenable graphs it is expected that there is a unique Gibbs state at $\beta_c$ . Let us formulate the question in term of FK-Ising. The celebrated argument of Benjamini, Peres, Lyons and Schramm regarding critical percolation on non-amenable graphs \cite{benjamini1999critical} works for FK-Ising  and it is proved in \cite{haggstrom2002explicit} that  $\phi^0_{\sfG, \beta_c} [0 \longleftrightarrow \infty] = 0$. It is expected that the phase transition is continuous for FK-Ising.
\begin{question} \label{q:nonamcr}
Let $\sfG$ be a non-amenable transitive graph, show that $\phi^1_{\sfG, \beta_c} [0 \longleftrightarrow \infty] = 0$?
\end{question}
Another interesting question is to identify translation-invariant Gibbs states of the $q$-state Potts model.
Here we do not provide the definitions, and we refer the reader to \cite{Gri06, duminil2017lectures}.
It is not true that the number of ergodic translation-invariant states of the $q$-state Potts model  is either $1$ or $q$: it is proved in \cite{biskup2003rigorous}  that for $q \geq 3$ on $\bbZ^d$, for $d$ high enough, there exists $\beta$ for which there exist at least $q+1$ ergodic states. 
See also \cite{duminil2016discontinuity} for a related result in the case of $\bbZ^2$ and $\beta= \beta_c$, and also \cite{KotShl82}.
For $\bbZ^2$ in \cite{coquille2014gibbs}, the set of all the Gibbs states is identified for any $\beta \neq \beta_c$. 
The following question is only a representative of an array of interesting questions.
\begin{question}
Consider the $q$-state Potts model on $\bbZ^d$ with $d \geq 3$ and $\beta > \beta_c$. Is it true that the set of ergodic translation-invariant Gibbs states has exactly $q$ elements?
\end{question}

If we replace the factor $2^k$ in the definition of the FK-Ising measure with $q^k$ we obtain the random-cluster model with parameter $q$. Just like the case of $q=2$, the measure is related to $q$-state Potts model for any integer $q$. The random-cluster measure is also defined for non-integer values of $q$ in the same manner. It is natural to ask the following question.
\begin{question} \label{q:fkq}
Let $ 1 \leq q \leq 2$. Let $\sfG$ be an amenable transitive graph. Is it true that for any $\beta \in [0, \infty]$ the set of random-cluster infinite-volume measures with parameter $q$ is a singleton?
\end{question}

\bibliographystyle{alpha}
\bibliography{biblicomplete}
\small\begin{flushright}
\textsc{Institut des Hautes \'Etudes Scientifiques \\}
\textsc{Bures-sur-Yvette, France \\}
\textsc{E-mail:} \texttt{raoufi@ihes.fr}
\end{flushright}
\end{document}